\documentclass[12pt]{amsart}

\usepackage{amsmath, amssymb, amscd, newlfont}
\usepackage{comment}

\newcommand{\kv}{{k_{v}}}

\newcommand{\bbA}{{\mathbb{A}}}
\newcommand{\bbF}{{\mathbb{F}}}

\newcommand{\bbQ}{{\mathbb{Q}}}
\newcommand{\bbC}{{\mathbb{C}}}

\newcommand{\bbR}{{\mathbb{R}}}
\newcommand{\bbZ}{{\mathbb{Z}}}

\newcommand{\supp}{{\mathrm{supp}}}

\newcommand{\Ind}{{\mathrm{Ind}}}

\newcommand{\GL}{{\mathrm{GL}}}

\newcommand{\triv}{{\mathbf{1}}}

\newcommand{\trace}{{\mathbf{tr}}}

\newcommand{\calA}{{\mathcal{A}}}

\newcommand{\calF}{{\mathcal{F}}}

\newcommand{\calK}{{\mathcal{K}}}
\newcommand{\calL}{{\mathcal{L}}}

\newcommand{\fraka}{{\mathfrak{a}}}

\usepackage[mathscr]{eucal}
\def \ScptB{\mathcal B}
\def \ScptG{\mathcal G}
\def \ScptK{\mathcal K}
\def \ScptP{\mathcal P}
\def \ScptU{\mathcal U}
\def \mymod{\text{\rm{{\ }mod{\ }}}}
\def \mystrongdivide{|_s}
\def \myskip{\vskip 0.2in}
\def\GL#1{{\text{\rm{GL}}}_{#1}}
\def\SL#1{{\text{\rm{SL}}}_{#1}}

\def \fkg{{\mathfrak g}}
\def \fkl{{\mathfrak l}}
\def \fkm{{\mathfrak m}}
\def \fkn{{\mathfrak n}}
\def \fkp{{\mathfrak p}}
\def \fks{{\mathfrak s}}
\def \fku{{\mathfrak u}}

\font\sans=cmss10

\numberwithin{equation}{section}

\newtheorem{Prop}[equation]{Proposition}
\newtheorem{Lem}[equation]{Lemma}

\newtheorem{Thm}[equation] {Theorem}
\newtheorem{Cor}[equation]{Corollary}

\newtheorem{Assumptions}[equation]{Assumptions}

\title
[
Cusp Forms  of Congruence Subgroups
]
{On the Cusp Forms of Congruence Subgroups of an almost Simple Lie group} 

\author{Allen Moy and Goran Mui\'c}

\address{ Department of Mathematics,
The Hong--Kong University of Science and Technology,
Clear Water Bay, Hong Kong}
\email{amoy@ust.hk}

\address{ Department of Mathematics,
University of Zagreb,
Bijeni\v cka 30, 10000 Zagreb,
Croatia}
\email{gmuic@math.hr}

\thanks{The 1st author acknowledges Hong Kong Research Grants Council
grant CERG {\#}603310, and the 2nd author acknowledges Croatian Ministry of Science and Technology grant {\#}0037108.}

\begin{document}
\maketitle

\begin{abstract}

In this paper we address the issue of existence of newforms among the
cusp forms for almost simple Lie groups using the approach of the second author combined with local information on supercuspidal representations for $p$-adic groups known by the first author. We pay special attention to the case of $SL_M(\Bbb R)$ where we prove various existence results for principal congruence subgroups.

\end{abstract}

\section{Introduction}


The existence and construction of cusp forms is a fundamental problem 
in the modern theory of automorphic forms (\cite{arthur}, \cite{sel}, 
\cite{wmuller1}, \cite{gold}, \cite{gold-1}).   
In this paper we address the issue of existence of cusp forms for almost simple Lie groups using the approach of \cite{Muic1} combined with some 
local information on supercuspidal representations for $p$-adic groups 
(\cite{MoyPr1}, \cite{MoyPr}). In view of recent developments in the analytic number theory (\cite{gold}, \cite{gold-1}) we pay special attention to the case of $SL_M$. 

\medskip

Suppose $G$ is a simply connected, absolutely almost simple algebraic group defined over $\bbQ$, and $G_\infty := G(\bbR )$ is not compact.   Let $\bbA$ and $\bbA_f$ denote respectively the ring of adeles and finite adeles of  $\Bbb Q$.  For each prime $p$, let $\bbZ_p$ denote the p-adic integers inside $\bbQ_p$.  Recall that for almost all primes $p$, the group $G$ is unramified over $\bbQ_p$.  Thus, $G$ is a group scheme over $\bbZ_p$, and   $G(\bbZ_p)$ is a hyperspecial maximal compact subgroup of  $G(\bbQ_p)$ (\cite{Tits}, 3.9.1). The group $G(\bbA_f)$ has a basis of neighborhoods of the identity consisting of open-compact subgroups.   Suppose $L \subset G(\bbA_f)$ is an open--compact subgroup.  Set
\begin{equation}\label{int-1}
\Gamma_{L} \ := \ G(\bbQ) \cap L \subset G(\bbA_f),
\end{equation}
where we identify $G(\Bbb Q)$ with its image under the diagonal embedding
into $G(\bbA_f)$ and $G(\bbA )$.  Now, 
identifying $G(\Bbb Q)$ with its image under the diagonal embedding
into $G(\bbA)$, the projection of $\Gamma_{L} \subset G(\bbA )$ to  $G_\infty$ is a discrete subgroup. We continue to denote this discrete subgroup
by $\Gamma_{L}$. It is called a congruence subgroup of $G( \bbQ )$ (\cite{BJ}).
We write $\calA_{cusp}(\Gamma_{L}\backslash G_\infty)$ and $L^2_{cusp}(\Gamma_{L} \backslash G_\infty)$ for the spaces of cusp forms and its $L^2$-closure (\cite{BJ}). 
We recall the notion of $\Gamma_{L}$--cuspidality here: a continuous function $\varphi: \ \Gamma_{L}\backslash G_\infty\rightarrow \mathbb C$ is $\Gamma_{L}$--cuspidal if 
$$
\int_{U_P(\mathbb R)\cap \Gamma_{L}\backslash U_P(\mathbb R)}\varphi (ug)=0, \ \ g\in G_\infty,
$$
where $U_P$ is the unipotent radical of any proper $\mathbb Q$--parabolic subgroup.
\myskip

Recall the assumptions that $G$ is simply connected, absolutely almost simple, and $G_{\infty}$ is non-compact means it satisfies the strong approximation property (\cite{Platonov}, \cite{BJ} \S 4.7), i.e., $G(\Bbb Q)$ is dense in $G(\Bbb A_f )$, and so for any open compact subgroup  $L \subset G(\Bbb A_f)$:
$$
G(\Bbb A_f)=G(\Bbb Q) L \ .
$$

\noindent We consider a finite family of open compact subgroups
\begin{equation}\label{finite-family}
{\mathcal F} = \{ L \}
\end{equation}
satisfying the following assumptions:

\medskip


\begin{Assumptions}\label{assumptions} \ 

\smallskip
\begin{itemize}
\item[(i)] Under the partial ordering of inclusion there exists a subgroup 
$L_{\text{\rm{min}}} \in {\mathcal F}$ that is a subgroup of all the others.  
\smallskip 
\item[(ii)] The groups $L\in {\mathcal F}$ are factorizable, i.e., $L={\underset p \prod} L_p$, and for all but finitely many $p$'s, the group $L_p$ is the maximal compact subgroup $K_p := G(\bbZ_p)$. 
\smallskip 
\item[(iii)]  There exists a non-empty finite set of primes $T$ such that for $p\in T$  the group $G$ has a Borel subgroup $B=AU$ and a maximal torus $A$ defined over $\Bbb Z_p$, and there exists a supercuspidal representation $\pi_p$ of  $G(\Bbb Q_p)$ such that $\pi^{L_{\text{\rm{min}}, p}}_p\neq 0$, and for $L\neq L_{\text{\rm{min}}}$ there exists $p\in T$ such that $\pi^{L_p}_p=0$.
\end{itemize}
\end{Assumptions}

\noindent We note that a simply connected split almost simple group, e.g.,  $SL_M$, and $Sp_{2M}$, defined over ${\Bbb Z}$ satisfies the above assumptions.

\bigskip

\begin{Thm}\label{intr-thm}  Suppose $G$ is a simply connected, absolutely almost simple algebraic group defined over ${\Bbb Q}$, such that $G_{\infty}$ is non-compact and ${\mathcal F} = \{ L \}$ is a finite set of open compact subgroups of $G({\bbA}_{f})$ satisfying assumptions \eqref{assumptions}.  Then, the orthogonal complement of 
$$
\sum_{\substack{ L \in {\mathcal F}  \\ L_{\text{\rm{min}}} \subsetneq L }} L^2_{cusp}(\Gamma_L \backslash G_\infty)
$$ 
in 
$L^2_{cusp}(\Gamma_{L_{\text{\rm{min}}}} \backslash G_\infty)$ is a direct sum of 
infinitely many irreducible unitary representations of $G_\infty$.
\end{Thm}

\bigskip

Theorem \ref{intr-thm} is proved in Section \ref{sec-1}. 
For general $G$, in Sections \ref{sec-2} and \ref{sec-3}
we give examples of families of ${\mathcal F}$ satisfying assumption \eqref{assumptions} using Moy--Prasad filtration subgroups 
(\cite{MoyPr1}, \cite{MoyPr}).

\bigskip

In the case $G=SL_M$, for suitable open compact subgroups $L \subset G(\bbA_{f} )$, the congruence subgroups $\Gamma_{L}$ are the principal congruence subgroups $\Gamma(m)$ (see (\ref{principal-congruence})), and the main theorem has 
the following form:

\begin{Cor}\label{intr-thm-cor}  
Let $G=SL_M$. Let $n\ge 2$ be an integer. Then,  
the orthogonal complement of 
$$
\sum_{\substack{m| n\\ m< n }} L^2_{cusp}(\Gamma(m)\backslash G_\infty)
$$ 
in 
$L^2_{cusp}(\Gamma(n)\backslash G_\infty)$ is a direct sum of  
infinitely many irreducible unitary representations of $G_\infty$.
\end{Cor}
\begin{proof} This follows directly from the examples in Section \ref{sec-3}.
\end{proof}

\vskip .2in 
In Section  \ref{sec-4}, we refine Corollary \ref{intr-thm-cor} (see Theorem \ref{nthm}). 
As a result, we obtain a generalization of  the compact quotient case (obtained
in \cite{Muic3}). The corresponding results are contained in Corollaries \ref{ncor-1}, 
\ref{ncor-2}, and \ref{ncor-3}. For example, in Corollary \ref{ncor-1}, 
we prove for sufficiently  large $n$ that we can take infinitely many 
spherical representations. Corollary \ref{ncor-3} improves (\cite{Muic2}, Theorem 0-2).

\bigskip

The initial research and first draft of the paper was performed when the second author visited The Hong Kong University of Science and Technology Mathematics Department in Spring 2010.  Final work and writing of the paper was done during a visit by the first author to the Mathematics Department of the University of Zagreb in Spring 2012.  The authors thank the Departments for their hospitality.   

\bigskip

\section{Proof of Theorem \ref{intr-thm}}\label{sec-1}

\bigskip

We recall  some results from \cite{Muic1}.  For $f \in C^\infty_c(G(\bbA))$, the adelic compactly supported Poincar{\'{e}} series  $ P(f)$ is 
defined as:
\begin{equation}\label{e-1}
P(f)(g) \ = \ \sum_{\gamma\in G(\bbQ)}  f(\gamma\cdot g).
\end{equation}

\noindent Write $g \in G(\bbA) = G_\infty\times G(\bbA_f)$ as $g=(g_\infty, g_f)$ . We have the following:
\begin{equation}\label{adel-arch-CCC-PPP}
P(f)(g_\infty, 1)=\sum_{\gamma\in G(\bbQ )} f (\gamma \cdot g_\infty, \gamma).
\end{equation}

\noindent The next lemma (\cite{Muic1}, Proposition 3.2) describes the restriction of the Poincar{\'{e}} series (\ref{e-1}) to $G_\infty$.  

\begin{Lem}\label{lem-1} Let $f \in C_c^\infty(G(\bbA))$. 
Assume that $L$ is an open compact subgroup of  $G(\bbA_f)$ such that $f$ is right--invariant under $L$.  Define the congruence subgroup $\Gamma_{L}$ of $G_\infty$ as in (\ref{int-1}).  Then:
\begin{itemize}
\item[(i)] The function in (\ref{adel-arch-CCC-PPP}) is a compactly supported Poincar{\'{e}} series on $G_\infty$ for $\Gamma_L$. 
\smallskip
\item[(ii)] If $P(f)$ is cuspidal, then  the function in (\ref{adel-arch-CCC-PPP}) is cuspidal for $\Gamma_L$ (see the Introduction for the definition).
\end{itemize}
\end{Lem}

We recall that $P(f)$ is cuspidal if 
$$
\int_{U_P(\mathbb Q)\backslash U_P(\mathbb A)}P(f)(ug)du=0, \ \ g\in G(\mathbb A),
$$
where $U_P$ is the unipotent radical of any proper $\mathbb Q$--parabolic subgroup.

Let $S$ be a finite set of places, containing $\infty$, and large enough so that $G$ is defined over $\Bbb Z_p$  for $p\not\in S$.   We use the decomposition of $G(\bbA)$ given by: 
\begin{equation}\label{e-2}
G(\bbA)=G_S\times G^S \, , {\text{\rm{ where }}} G_S := {\underset {p\in S} \prod} G(\bbQ_p) \, , {\text{\rm{ and }}} G^S={\underset {p \not\in S} {{\prod}\,'}} G(\bbQ_p) \ .
\end{equation}  
\noindent Set $G^S(\bbZ_p) :={\underset {p \not\in S} \prod} G(\bbZ_p)$, and 
\begin{equation}\label{e-200}
\Gamma (S) \ := \ G^S(\bbZ_p) \cap G(\bbQ) \ \ \text{(the intersection is taken in $G^S$).}
\end{equation}
We view $\Gamma (S) \subset G(\bbQ ) \subset G(\bbA )$.  Set
$$
\Gamma_{S} \ = \ {\text{\rm{image of $\Gamma (S)$ under the projection map $G(\bbA ) = G_S \times G^S \rightarrow G_S$}}}\, .
$$ 
Since $G(\bbQ)$ is a discrete subgroup of $G(\bbA)$, it follows  that $\Gamma_S$ is a discrete subgroup of $G_S$.

\bigskip

Set 
\begin{equation}\label{finte-places-of-S}
S_{f} \ := \ {\text{\rm{ the set of finite places in $S$.}}}
\end{equation} 

\noindent For each $p\in S_{f}$, we choose an open--compact subgroup $L_p$, and we set 
\begin{equation}\label{e-20}
\begin{aligned}
L \ &= \ \Big( \, {\underset {p\in S_{f}} \prod} L_p \ \Big) \  \times \  G^S(\bbZ_{p})\\
\Gamma_{L} \ &= \ L\cap G(\bbQ) \ = \ \Big( \,  \prod_{p\in S_{f}} L_{p} \ \Big) \ \cap \ \Gamma_S 
  \ .
\end{aligned}
\end{equation}
The group $\Gamma_{L}$ is a discrete subgroup of $G_\infty$.

\medskip

Let $\mathfrak g_\infty$ be the (real) Lie algebra of $G_{\infty}$, and $K_\infty$  a maximal compact subgroup.  We have the following non-vanishing criterion (\cite{Muic1}, Theorem 4.2):

\medskip

\begin{Lem}\label{lem-2} 
Assume that for each prime $p$ we have a function 
$f_{p} \in C_c^\infty(G(\bbQ_{p}))$ so that $f_{p}(1)\neq 0$, and 
 $f_{p}=1_{G(\bbZ_{p})}$ is the characteristic function of $G(\bbZ_{p})$ for all $p\not\in S$.  Assume further that for $p \in S_{f}$, that $L_p$ is an open-compact subgroup such that $f_p$ is right-invariant under $L_p$.  Note:  Since the set $\big( \, K_\infty \times {\underset {p\in S_{f}} \prod} \supp{\ (f_p)} \, \big)$ is compact, the intersection
$\Gamma_S \cap \big( \, K_\infty \times {\underset {p\in S_{f}} \prod} \supp{\ (f_p)} \, \big)$ is a finite set.  It can be written
as follows: 
\begin{equation}\label{e-3}
{\underset {j=1} {\overset l \bigcup}} \ \ \gamma_j \cdot (K_\infty\cap \Gamma_{L}).
\end{equation}
Set 
$$
c_j \ = \ {\underset {p\in S_{f}} \prod} \ f_{p}(\gamma_j).
$$
Then, the $K_\infty$-invariant map
$C^\infty(K_\infty)\longrightarrow C^\infty(K_\infty\cap \Gamma_{L}\setminus K_\infty)$ given by 
\begin{equation}\label{e-4}
\alpha  \ \mapsto \ \hat{\alpha}(k) \ := \ k \ \mapsto \ \sum_{j=1}^l\sum_{\gamma\in K_\infty\cap \Gamma} \ c_j\cdot
\alpha(\gamma_j \gamma \cdot k) 
\end{equation}
is non--trivial, and, for every $\delta\in \hat{K}_\infty$, contributing to
the decomposition of the closure of the image of (\ref{e-4}) 
in  $L^2(K_\infty\cap \Gamma_{L} \setminus K_\infty)$, we can find a non-trivial $f_\infty \in
C_c^\infty(G_\infty)$ so that the following hold:
\begin{itemize}
\item[(i)] $E_\delta(f_\infty)=f_\infty$. 
\item[(ii)] The Poincar\' e series $P(f)$ and its restriction to
 $G_\infty$ (which is a Poincar\' e series for $\Gamma_L$) are 
 non--trivial, where $f \, := \, f_\infty\otimes_{p} f_{p}  \in
C_c^\infty(G(\bbA))$. 
\item[(iii)]  $E_\delta(P(f))=P(f)$ and $P(f)$ is right--invariant under $L$.
\item[(iv)] The support of $P(f)|_{G_\infty}$ is contained in a set of the form $\Gamma_{L}\cdot C$, where 
$C$ is a compact set which is right--invariant under 
$K_\infty$, and  $\Gamma_{L}\cdot C$ is not the whole of $G_\infty$. 
\end{itemize}
\end{Lem}

\vskip .2in 
We begin the proof of Theorem \ref{intr-thm}.  We apply the above considerations to $L_{\text{\rm{min}}}$.  By hypothesis, this group is factorizable.  Take $S$ sufficiently large so that it contains $T$, and if $p\not\in S$ then the group $G$ is unramified over $\Bbb Q_p$ so that it is defined over $\Bbb Z_p$, and $L_{\text{\rm{min}}, p}=G(\Bbb Z_p)$. Thus, (\ref{e-20}) holds for $L=L_{\text{\rm{min}}}$. We apply Lemma \ref{lem-2}.  To do this, we construct functions $f_{p} \in C_c^\infty(G(\bbQ_{p}))$ such that $f_{p}(1)\neq 0$, $f_p$ is $L_p$-invariant on the right, and $f_{p}=1_{G(\bbZ_{p})}$ for all $p\not\in S$.  We need to define $f_p$ for $p\in S_{f}$.  We let $f_p=1_{L_{\text{\rm{min}}, p}}$ for $p \in S_{f} \, \backslash \, T$.  For $p\in T$, we use our assumption that there exists a supercuspidal representation $\pi_p$ such that $\pi^{L_{\text{\rm{min}}, p}}_p\neq 0$.  We let 
$f_p$ be a matrix coefficient of  $\pi_p$ such that  $f_{p}(1)\neq 0$ and 
$f_p$ is $L_{\text{\rm{min}}, p}$-invariant on the right.  

\medskip

Our construction of the functions $f_p$ for all finite $p$ satisfy the assumptions of Lemma \ref{lem-2}.  Therefore, by that Lemma, we can select $\delta\in \hat{K}_\infty$ and $f_\infty \in C_c^\infty(G_\infty)$ such that (i)---(iv) of Lemma \ref{lem-2} hold. 

\vskip .2in

We decompose the cuspidal part $L_{cusp}^2(G(k)\setminus G(\bbA))$ into closed irreducible $G(\bbA)$ invariant subspaces: 
\begin{equation}\label{e-50}
L_{cusp}^2(G(k)\setminus G(\bbA))=\oplus_j \  {\mathfrak H}_j \ .
\end{equation}

We prove the following lemma: 

\begin{Lem}\label{lem-3}  We keep the above assumptions. 
Then,
$$
P(f)\in  L_{cusp}^2(G(\bbQ)\setminus G(\bbA))^{L_{\text{\rm{min}}}},
$$ 
and we can decompose according to (\ref{e-50})
\begin{equation}\label{e-5}
P(f)= \sum_j \psi_j, \ \ \psi_j\in {\mathfrak H}_j \ . 
\end{equation}
Then, we have the following:
\begin{itemize} 
\item [(i)] For all $j$, $\psi_j\in \cal
  A_{cusp}(G(\bbQ)\setminus G(\bbA))$ is right--invariant under $L_{\text{\rm{min}}}$,
  and transforms according to $\delta$ i.e.,
  $E_\delta(\psi_j)=\psi_j$. 
\item [(ii)]  Assume $\psi_j\neq 0$. Then  $\pi_{p}^j\simeq \pi_{p}$ for all
   $p \in T$.
\item [(iii)]  The number of indices $j$ in (\ref{e-5}) such that $\psi_j\neq 0$ is
infinite. 
\item [(iv)] The closure of the $G_\infty$--invariant subspace in 
$L^2_{cusp}(\Gamma_{L_{\text{\rm{min}}}}\setminus G_\infty)$ generated by $P(f)|_{G_\infty}$ is an orthogonal direct sum 
of infinitely many inequivalent irreducible unitary representations of $G_\infty$ 
which contain $\delta$.
\end{itemize}
\end{Lem}
\begin{proof} This follows from (\cite{Muic1}, Proposition 5.3 and Theorem 7.2). We remark that 
the formulation of (iv) follows from the proof of 
(\cite{Muic1}, Theorem 7.2). 
\end{proof}

\vskip .2in

\begin{Lem}\label{lem-4} Let $L\in \mathcal F$. Then, the restriction map gives an isomorphism of
unitary $G_\infty$ representations
$$
L^2_{cusp}(G(\bbQ)\setminus G(\bbA))^L\simeq  L^2_{cusp}(\Gamma_{L}\setminus G_\infty),
 $$
which is norm preserving up to a scalar $vol_{G(\bbA_f)}(L)$ i.e., 
$$
\int_{G(\bbQ)\setminus G(\bbA)} |\psi(g)|^2 dg=vol_{G(\bbA_f)}(L)
\int_{\Gamma_{L}\setminus G_\infty} |\psi(g_\infty)|^2 dg_\infty.
$$
\end{Lem}
\begin{proof} Since we assume that $G$ is simply connected, absolutely almost simple over $\bbQ$ and $G_\infty$ is not 
compact,  it satisfies the strong approximation property:  we have $G(\Bbb A_f)=G(\Bbb Q) L$. 
Now, the claim  is a particular case of the  computations contained in the proof of (\cite{Muic1}, Theorem 7-2, see 
(7-6) there). 
\end{proof}

\vskip .2in

\begin{Lem}\label{lem-5} Let $L\in \mathcal F$. Then, the natural embedding is an embedding 
of unitary $G_\infty$--representations:
$$
L^2_{cusp}(\Gamma_{L}\setminus G_\infty) \hookrightarrow L^2_{cusp}(\Gamma_{L_{\text{\rm{min}}}}\setminus G_\infty).
$$
This means it is norm preserving up to a scalar $\left[
\Gamma_L:\Gamma_{L_{\text{\rm{min}}}}\right]$ i.e., 
$$
\int_{\Gamma_{L_{\text{\rm{min}}}}\setminus G_\infty} |\psi(g_\infty)|^2 dg_\infty=\left[
\Gamma_L:\Gamma_{L_{\text{\rm{min}}}}\right]
\int_{\Gamma_{L}\setminus G_\infty} |\psi(g_\infty)|^2 dg_\infty,
$$
for $\psi \in L^2_{cusp}(\Gamma_{L}\setminus G_\infty)$.
\end{Lem}
\begin{proof} This is \cite{Muic2}, Lemma 1-9. 
\end{proof}

\vskip .2in

\begin{Lem}\label{lem-6} Let $L\in \mathcal F$. Then, we have
$$
\left[
\Gamma_L:\Gamma_{L_{\text{\rm{min}}}}\right]=vol_{G(\bbA_f)}(L)/vol_{G(\bbA_f)}(L_{\text{\rm{min}}}).
$$
\end{Lem}
\begin{proof} Obviously, we have 
$$
\left[L: L_{\text{\rm{min}}}\right]=vol_{G(\bbA_f)}(L)/vol_{G(\bbA_f)}(L_{\text{\rm{min}}}).
$$
But 
$$L/ L_{\text{\rm{min}}}= \left(L\cap G(\Bbb Q)\right)/\left(L_{\text{\rm{min}}}\cap G(\Bbb Q)\right)=
\Gamma_L/\Gamma_{L_{\text{\rm{min}}}},
$$
since
$$
G(\Bbb A_f)=G(\Bbb Q) L_{\text{\rm{min}}}.
$$
\end{proof}

\vskip .2in

\begin{Lem}\label{lem-7} 
We have the following commutative diagram  of unitary $G_\infty$--representations:
$$
\begin{CD}
L^2_{cusp}(G(\bbQ)\setminus G(\bbA))^L@>>>L^2_{cusp}(G(\bbQ)\setminus G(\bbA))^{L_{\text{\rm{min}}}}\\
@V\simeq VV@V\simeq VV\\
L^2_{cusp}(\Gamma_{L}\setminus G_\infty)
@>>>L^2_{cusp}(\Gamma_{L_{\text{\rm{min}}}}\setminus G_\infty),\\
\end{CD}
$$
where the horizontal maps are inclusions, the vertical maps are isomorphisms, and the 
inner products are normalized as follows: (i) on the spaces in the first row, we take the usual Petersson inner product $\int_{G(\bbQ)\setminus G(\bbA)} \psi(g)\overline{\varphi(g)} dg$, and (ii) on $L^2_{cusp}(\Gamma_{U}\setminus G_\infty)$, the inner product is the normalized integral  
$vol_{G(\bbA_f)}(U)\int_{\Gamma_{U}\setminus G_\infty} \psi(g_\infty)\overline{\varphi(g_\infty)} dg$,
where $U\in \{L_{\text{\rm{min}}}, L\}$.
\end{Lem}
\begin{proof} The lemma follows immediately from Lemmas \ref{lem-4}, \ref{lem-5}, and \ref{lem-6}.
\end{proof}

\vskip .2in

\begin{Lem}\label{lem-8} Let $L\in \mathcal F$, $L\neq L_{\text{\rm{min}}}$. Then,  
$P(f)|_{G_\infty}$ is orthogonal to 
$L^2_{cusp}(\Gamma_{L}\setminus G_\infty)$ in  $L^2_{cusp}(\Gamma_{L_{\text{\rm{min}}}}\setminus G_\infty)$
if and only if $P(f)$ is orthogonal to $L^2_{cusp}(G(\bbQ)\setminus G(\bbA))^L$.
\end{Lem}
\begin{proof} This follows from Lemma  \ref{lem-7}.
\end{proof} 

\vskip .2in 
By Lemma \ref{lem-3} (iv), the closure $\cal U$ of the $G_\infty$-invariant subspace in 
$L^2_{cusp}(\Gamma_{L_{\text{\rm{min}}}}\setminus G_\infty)$ generated by $P(f)|_{G_\infty}$ is an orthogonal direct sum of 
infinitely many inequivalent irreducible unitary representations of $G_\infty$. By Lemma \ref{lem-8}, 
$\cal U$ is orthogonal to $$
\sum_{\substack{L\in \mathcal F\\
L\neq L_{\text{\rm{min}}}}}\ \  L^2_{cusp}(\Gamma_{L}\setminus G_\infty)
$$
if and only if $P(f)$ is orthogonal to $L^2_{cusp}(G(\bbQ)\setminus G(\bbA))^L$
for all $L\in \mathcal F$, $L\neq L_{\text{\rm{min}}}$. The following lemma
completes the proof of Theorem \ref{intr-thm}.

\vskip .2in 
\begin{Lem}\label{lem-9} Let $L\in \mathcal F$, $L\neq L_{\text{\rm{min}}}$. Then,  
$P(f)$ is orthogonal to $L^2_{cusp}(G(\bbQ)\setminus G(\bbA))^L$.
\end{Lem}
\begin{proof} Here we use the very last assumption that $\pi^{L_p}=0$ for all 
$L\in {\mathcal F}$ such that $L\neq L_{\text{\rm{min}}}$.

We remind the reader that $f_p$ is a matrix 
coefficient of $\pi_p$ for $p\in T$. 
 Since $ L_{\text{\rm{min}}}\subset L$, we have 
$ L_{\text{\rm{min}}, p }\subset L_p$ for all primes $p$  (the groups are factorizable), 
we obtain that 
$$
\int_{L} f(g l) dl =0, \ \ \text{for all $g\in G(\bbA)$.}
$$
Hence 
$$
\int_{L} P(f)(g l) dl =0, \ \ \text{for all $g\in G(\bbA)$.}
$$

Finally, let $\varphi \in L_{cusp}^2(G(\bbQ)\setminus G(\bbA))^{L}$. Then
we compute
\begin{align*}
& vol_{G(\bbA_f)}(L)\cdot \int_{G(\bbQ)\setminus G(\bbA)} P(f)(g)\overline{\varphi(g)}dg=\\
& \int_{G(\bbQ)\setminus G(\bbA)} \int_{L} P(f)(gl)\overline{\varphi(gl)}dldg =\\
&  \int_{G(\bbQ)\setminus G(\bbA)} \left(\int_{L} P(f)(gl)dl\right) \overline{\varphi(g)}dg=0.
\end{align*}
This proves the lemma. 
\end{proof}

\bigskip

\bigskip

\section{Open compact subgroups of $G({\bbQ}_p)$, congruence subgroups, and cusp forms}\label{sec-2}

\medskip

\subsection{The case of SL$_M$}\label{case-slm} \ 

\smallskip

Fix a positive integer $M$, and consider the algebraic reductive group $G=\SL{M}$.   In $\SL{M}({\mathbb Z})$, we consider an alternative to the usual principal congruence subgroup 
\begin{equation}\label{principal-congruence}
\Gamma (n) \ := \ \{ \, g = (g_{i,j}) \in \SL{M}({\mathbb Z}) \ | \ g_{i,j} \equiv \delta_{i,j} \mymod n \, \} \ .
\end{equation}

\noindent For a prime power $p^k$ ($k > 0$), we set 

\begin{equation}\label{iwahori-congruence}
\aligned
\Gamma_{1} (p^k) \ := \ \{ \, g = (g_{i,j}) \in \SL{M}({\mathbb Z}) \ | \ p^{(k-1)} \, &| \, g_{i,j} {\text{\rm{ for $i<j$}}} \\
p^k \, &| \, (g_{i,j} - \delta_{i,j}) {\text{\rm{ for $i \ge j$}}} \ \} .
\endaligned
\end{equation}

\noindent $\Gamma_{1} (p^k)$ is the set of elements in $\Gamma (p^{(k-1)})$ which modulo $p^k$ are unipotent upper triangular.    If $n$ is a positive integer, with prime factorization $n={p_{1}}^{e_{1}} \cdots  {p_{s}}^{e_{s}}$, set 

\begin{equation}
\Gamma_{1} (n) \ := \ {\underset {i} \bigcap } \ \Gamma_{1} ({p_{i}}^{e_{i}}) \ .
\end{equation}
\bigskip

\noindent Recall that $\Gamma (n)$ is a subgroup of $\Gamma (m)$ if and only if $m$ divides $n$.  Moreover, when $\Gamma (n)$ is a subgroup of $\Gamma (m)$, it is a normal subgroup.  The situation for the groups $\Gamma_{1}(m)$ is slightly more complicate.  We note that if $m$ divides $n$, then $\Gamma_{1}(n)$ is a subgroup of $\Gamma_{1}(m)$.  However, for $n > 1$, we also note  $\Gamma_{1} (n)$ is not a normal subgroup of $\Gamma (1) = \SL{M}({\mathbb Z})$, and that $\Gamma_{1}(n)$ is not necessarily a normal subgroup of $\Gamma_{1}(m)$ when $m | n$.   Define $n$ to be a strong multiple of $m$ (or $m$ divides $n$ strongly) if $n$ is a multiple of $m$ and every prime $p$ that occurs in the factorization of $n$ also occurs in the factorization of $m$.  We use the notation $m \mystrongdivide n$.  The following elementary result then provides a criterion for when $\Gamma_1(n)$ is a normal subgroup of $\Gamma_{1}(m)$.

\medskip

\begin{Prop}\label{gamma-one-1}  
\begin{itemize}
\item[(i)] For $k \ge 1$, the group $\Gamma_{1}(p^{k})$ is a normal subgroup of $\Gamma_{1}(p)$.
\smallskip
\item[(ii)] If $m \mystrongdivide n$, then $\Gamma_{1}(n)$ is a normal subgroup of $\Gamma_{1}(m)$.
\end{itemize}
\end{Prop}

\smallskip

For $k$ a positive integer, define open compact subgroups of $\SL{M}({\mathbb Q}_p)$ as follows:

\begin{equation}\label{iwahori-k}
\aligned
{\mathcal K}_{p,k} \ :&= \ \{ \, (g_{i,j}) \in \SL{M}({\mathbb Z}_{p}) \ \big| \ p^{k} \, | \, (g_{i,j} - \delta_{i,j}) \ \forall \, i,j \ \} \ , \ {\text{\rm{ and }}} \\
&\ \\
{\mathcal I}_{p,k^{+}} \ :&= \ \{ \, (g_{i,j}) \in \SL{M}({\mathbb Z}_{p}) \ \big| \ p^{(k-1)} \, | \, g_{i,j} {\text{\rm{ for $i<j$,}}} \ {\text{\rm{and}}}  \ p^k \, | \, (g_{i,j} - \delta_{i,j}) {\text{\rm{ for $i \ge j$}}} \ \} .
\endaligned
\end{equation}
\noindent The group ${\mathcal K}_{p,k}$ is the well-known k-th congruence subgroup of $\SL{M}({\mathbb Z}_{p})$.  Set 

\begin{equation}\label{open-compact-SL}
\aligned
K_n \, :&= \ {\underset {p_i | n} \prod } \  {\mathcal K}_{p_i, e_i} \ \ {\underset {p {\nmid} n} \prod } \ G({\mathbb Z}_p) \ , \ \ {\text{\rm{and}}} \ \ I_n \, := \ {\underset {p_i | n} \prod } \  {\mathcal I}_{p_i, (e_i)^{+}} \ \ {\underset {p {\nmid} n} \prod } \ G({\mathbb Z}_p)  \ .
\endaligned
\end{equation}

\smallskip

\noindent A consequence of the Chinese remainder theorem is the following:

\begin{Prop}\label{gamma-one-2}  

\end{Prop}
$$
\aligned
\Gamma (n) \ = \ K_n \ \cap \ G( {\mathbb Q}) \, , \ \ {\text{\rm{and}}} \ \ \Gamma_1 (n) \ = \ I_n \ \cap \ G( {\mathbb Q}) \ .
\endaligned
$$
\smallskip

\noindent The groups ${\mathcal I}_{p,k^{+}}$ are the same as certain Moy-Prasad filtration subgroups.   Let ${\mathcal I}_{p}$ denote the subgroup consisting of elements in $\SL{M}({\mathbb Z}_{p})$ which are upper triangular modulo $p$.  It is an Iwahori subgroup of $\SL{M}({\mathbb Q}_{p})$. Let $\ScptB (\SL{M}({\mathbb Q}_{p}))$ be the Bruhat-Tits building of $\SL{M}({\mathbb Q}_{p})$.  Let $C$ be the alcove in $\ScptB (\SL{M}({\mathbb Q}_{p}))$ fixed by ${\mathcal I}_{p}$.   Then, for any $x \in C$, and $k \in {\mathbb N}$, we have  ${\mathcal I}_{p,k^{+}} = {\big(}{\SL{M}({\mathbb Q}_{p})}{\big)}_{x,k^{+}}$. 

\medskip

We prove a cusp form result related to the group ${\mathcal K}_{p,k}$.  As a preliminary, we recall the definition of cusp forms for a p-adic or finite field semi-simple group $\ScptG$ and its Lie algebra $\fkg$:  
\begin{itemize}
\item[(i)] A function $\phi \in C^{\infty}_{0} (\fkg )$ is a cusp form if for any $X \in \fkg$ and unipotent radical $\fkn$ of any proper parabolic subalgebra $\fkp \subsetneq \fkg$, the integral 
\begin{equation}\label{lie-algebra-cusp-form-1}
\int_{\fkn} \ \phi \, (X +n ) \ dn \ {\text{\rm{ is zero.}}}
\end{equation}
\item[(ii)] A function $f \in C^{\infty}_{0} (\ScptG )$ is a cusp form if for any $X,Y  \in \ScptG$ and unipotent radical $\ScptU$ of any proper parabolic subgroup $\ScptP \subsetneq \ScptG$, the integral 
\begin{equation}\label{lie-algebra-cusp-form-2}
\int_{\ScptU} \ f(X \, u \, Y) \ du \ {\text{\rm{ is zero.}}}
\end{equation}
\end{itemize}

As noted earlier, the k-th congruence subgroup ${\mathcal K}_{p,k}$ is normal in $\SL{M}({\bbZ}_p)$.  For $k >0$, we now formulate and prove a result on cusp forms associated to certain characters of the group ${\mathcal K}_{p,k}/{\mathcal K}_{p,(k+1)}$.  To simplify the notation (since $p$ will be fixed), we abbreviate ${\mathcal K}_{p,k}$ to ${\mathcal K}_{k}$.  Set ${\calL} = {\fks \fkl}_{M}({\bbZ}_p)$, and  
\begin{equation}
{\calL}_{k} \ := \ \{ \ (x_{i,j}) \in {\calL} \ \big| \ p^k \, | \, \, x_{i,j} \ \} \ = \ p^k {\calL} \ .
\end{equation}
The quotient ${\calL}/{\calL}_1$ is naturally the Lie algebra ${\fks \fkl}_{M}({\bbF}_p)$.  The map $X \rightarrow p^kX$ gives a natural isomorphism $\tau_{k}$ of ${\calL}/{\calL}_1$ with ${\calL}_k/{\calL}_{(k+1)}$
Recall there is an isomorphism  
\begin{equation} \label{cusp-K-one}
\theta \ : \ {\calL}_k/{\calL}_{(k+1)} \ \rightarrow \ {\mathcal K}_{k}/{\mathcal K}_{(k+1)}
\end{equation} 
so that for $x \in {\calL}_k$, that 
$$
\theta (x) \ = \ 1 + x {\text{\rm{ mod }}} {\calK}_{(k+1)}  {\text{\rm{ \ in \ $\GL{M} (\bbZ_p )$}}} \ .
$$
\noindent Therefore, there is a natural isomorphism of ${\mathcal K}_{k}/{\mathcal K}_{(k+1)}$ with ${\fks \fkl}_{M}({\bbF}_p)$.  Any choice of a non-trivial additive character $\psi$ of ${\bbF}_p$ gives an identification of the Pontryagin dual of ${\fks \fkl}_{M}({\bbF}_p)$ with ${\fkg \fkl}_{M}({\bbF}_p)/{\bbF}_p$, via the composition of the pairing
\begin{equation} \label{cusp-K-two}
\aligned
{\fks \fkl}_{M}({\bbF}_p) \ \times \ {\fkg \fkl}_{M}({\bbF}_p) \ &\longrightarrow \ {\bbF}_p \\
( \, X \, , \, Y \, ) \qquad &\longrightarrow \ \trace ( \, XY \, )  
\endaligned
\end{equation}
\noindent with $\psi$.  Write 
\begin{equation} \label{cusp-K-three}
\aligned
\psi_{Y} \ : \ {\fks \fkl}_{M}({\bbF}_p) \ \longrightarrow& \ {\bbC}^{\times} \\
X \ {\overset {\psi_{Y}} \longrightarrow}& \ \psi ( \, \trace ( \, XY \, ) \, ) \ .
\endaligned
\end{equation}

Take $Y \in {\fkg \fkl}_{M}({\bbF}_p)$ to be an element whose characteristic polynomial is irreducible.  Such an element exists since there is a finite extension of ${\bbF}_p$ of degree $M$.   The following proposition is elementary.

\begin{Prop} \label{SL-cusp-forms}
Suppose $Y \in {\fkg \fkl}_{M}({\bbF}_p)$ has irreducible characteristic polynomial:

\smallskip

\begin{itemize}
\item[(i)] If $\fkp ({\bbF}_p) \subsetneq {\fkg \fkl}_{M}({\bbF}_p)$ is any parabolic subalgebra of ${\fkg \fkl}_{M}({\bbF}_p)$, then $Y \notin \fkp ({\bbF}_p)$. 

\smallskip

\item[(ii)] The character $\psi_{Y}$ of ${\fks \fkl}_{M}({\bbF}_p)$ is a cusp form.

\smallskip

\item[(iii)] The inflation of $\psi_{Y}$ to ${\calK}_k$ via \eqref{cusp-K-one}, when extended to $\SL{M}(\bbQ_p)$ by setting it zero off  ${\calK}_k$ is a cusp form.

\smallskip

\item[(iv)] For each positive integer $k$, there exists an irreducible supercuspidal representation $(\rho ,W_{\rho})$ which has a non-zero ${\calK}_{k+1}$ fixed vector, but no non-zero ${\calK}_{k}$-fixed vector.
\end{itemize}
\end{Prop}

\smallskip

\begin{proof}  To prove part (i), suppose $\fkp ({\bbF}_p) \subsetneq {\fkg \fkl}_{M}({\bbF}_p)$ is any parabolic subalgebra and $\fkp ({\bbF}_p) = \fkm ({\bbF}_p) + \fku ({\bbF}_p)$ is a Levi decomposition.  For any $Z \in \fkp ({\bbF}_p)$, let $Z_{\fkm ({\bbF}_p)}$ be the projection of $Z$ to ${\fkm ({\bbF}_p)}$.  Then $Z$ and $Z_{\fkm ({\bbF}_p)}$ have the same characteristic polynomial, and the latter characteristic polynomial is clearly not irreducible.  Thus, if $Y \in {\fkg \fkl}_{M}({\bbF}_p)$ has irreducible characteristic polynomial, it cannot lie in any $\fkp ({\bbF}_p) \subsetneq {\fkg \fkl}_{M}({\bbF}_p)$.

\medskip

To prove (ii), suppose $x \in {\fks \fkl}_{M}({\bbF}_p)$, and $\fkp ({\bbF}_p) = \fkm ({\bbF}_p) + \fku ({\bbF}_p)$ is a proper parabolic subalgebra.  Then 
$$
\int_{\fku ({\bbF}_p)} \psi_{Y} (x+n) \, dn \ = \ \psi_{Y}(x) \ \int_{\fku ({\bbF}_p)} \psi_{Y} (n) \, dn
$$
Since $Y$ is not contained in any parabolic subalgebra of ${\fkg \fkl}_{M}({\bbF}_p)$, the integrand on the right side is a non-trivial character of ${\fku ({\bbF}_p)}$ and therefore the integral is zero.  Whence, $\psi_{Y}$ is a cusp form.

\medskip

To prove (iii), denote the inflation of $\psi_{Y}$ by $\psi_{Y} \circ \theta^{-1}$.  Suppose $P \subset \SL{M}({\bbQ}_p)$ is a parabolic subgroup.  Then $P$ is conjugate to a standard `block upper triangular' parabolic subgroup $Q=M_{Q}N_{Q}$ of $\SL{M}({\bbQ}_p)$, i.e., 
$$
P \ = \ g^{-1} Q g \ = \ ( \, g^{-1}M_{Q}g \, ) \ ( \, g^{-1}N_{Q}g \, ) \ {\text{\rm{ \ and \ $U_{P} = g^{-1}N_{Q}g$. }}}
$$ 
\noindent Since $\SL{M}({\bbQ}_p) = Q{\mathcal K}$, express $g$ as $g = v_{g} k_{g}$ with $v_{g} \in Q$, and $k_{g} \in {\mathcal K}$.  Then, 
$$
\aligned
\int_{U_{P}} \ \psi_{Y} \circ \theta^{-1}(xn) \ dn \ &= \ \int_{U_{Q}} \ \psi_{Y} \circ \theta^{-1}(xg^{-1}ug) \ du \\
&= \ \int_{U_{Q}} \ \psi_{Y} \circ \theta^{-1}(xk_{g}^{-1} v_{g}^{-1} u v_{g} k_{g}) \ du \\
&= \ c \ \int_{U_{Q}} \ \psi_{Y} \circ \theta^{-1}(xk_{g}^{-1}  u k_{g}) \ du \ \ {\text{\rm{(suitable constant $c$)}}}\\
&= \ c \ \int_{k_{g}^{-1}U_{Q}k_{g}} \ \psi_{Y} \circ \theta^{-1}(x u ) \ du \ \\
\endaligned
$$
\noindent From the last line, since  $\psi_{Y} \circ \theta^{-1}$ has support in ${\mathcal K}_k$, to prove the integral vanishes, it suffices to do so when $x \in {\mathcal K}_k$.  In this situation the integral vanishes by part (ii).  Thus $\psi_{Y} \circ \theta^{-1}$ is a cusp form on $\SL{M}({\bbQ}_p)$.

\medskip

In regards to part (iv), set $\chi = ( \psi_{Y} \circ \theta^{-1} )$, and let $V_{\chi}$ be the representation of $\ScptG = \SL{M}({\bbQ}_p)$ generated by the left translates of the cusp form $\chi$.  It is the induced representation ${\text{\rm{c-Ind}}}^{\ScptG}_{{\mathcal K}_{k}} ( \chi )$.  As a representation inside the unitary representation $L^{2}( \ScptG )$, $V_{\chi}$ is unitarizable and therefore completely reducible.     The Hecke algebra 
$$
\aligned
{\mathcal H}(\ScptG //{{\mathcal K}_{k}} \, , \,  \chi^{-1} ) \ :&= \ \{ \ f \in C^{\infty}_{0}( (\ScptG ) \ | \ \forall \, m_a,m_b \in K_{k} {\text{\rm{ and }}} g \in SL_{M}({\bbQ}_p) \ , \ \\
&\qquad \quad f(m_a \, g \, m_b) \, = \, \chi (m_a)^{-1} \, f(g) \, \chi (m_b)^{-1}) \ \} \ ,
\endaligned
$$
is the endomorphism algebra of the unitary representation  $V_{\chi}$.  It follows from the fact that $\chi$ is a cusp form and the Cartan decompostion $\ScptG = {\mathcal K} A^{+} {\mathcal K}$ ($A$ the subgroup of diagonal matrices), that there is a sufficiently large compact subset $C \subset \ScptG$ so that the support of any $f \in {\mathcal H}(\ScptG //{{\mathcal K}_{k}} \, , \,  \chi^{-1} )$ is contained in $C$.  The dimension of  ${\mathcal H}(\ScptG //{{\mathcal K}_{k}} \, , \,  \chi^{-1} )$ is therefore finite (see also \cite{harish-chandra} page 28), and consequently the (unitarizable) representation $V_{\chi}$ has finite length.  Whence,  $V_{\chi}$ is a finite direct sum of irreducible representations.  By Frobenius reciprocity, any irreducible subrepresentation $\sigma$ of ${\text{\rm{c-Ind}}}^{SL_{M}({\bbQ}_p)}_{{\mathcal K}_{k}} ( \chi )$ contains the character ${\chi}$.  Let $\langle , \rangle$  be the unitary form $\sigma$ inherits as a subrepresentation of $L^{2} (\ScptG )$.  Write $\lambda$ for the left translation action, and take a non-zero $v$ so that 
$$
\forall \ m \in {\mathcal K}_{k} : \quad \lambda (m) v \ = \ \chi (m) v \ .
$$
Consider the (non-zero) matrix coefficient $L_{v}$ defined as $L_{v}(g) = \langle \, \lambda (g) v \, , \, v \rangle$.  We claim, as a consequence of $\chi$ being 
a cusp form, that $L_{v}$ is a cusp form too.  Indeed, for any $X \in \ScptG$, and ${\mathcal U}$ the unipotent radical of any proper parabolic group ${\mathcal P} \subset \ScptG$, we have:

$$
\aligned
\int_{\ScptU \cap \ScptK_{k}} L_{v}(Xu) \ du &= \ \int_{\ScptU \cap \ScptK_{k}} \ \langle \, \lambda (X) \, \lambda (u) v \, , \, v \, \rangle \ du  \\
&= \  \langle \, \lambda (X) ( \, \int_{\ScptU \cap \ScptK_{k}} \lambda (u) v \, du \, ) \, , \, v \, \rangle  \ . \\
\endaligned 
$$
Since $( \int_{\ScptU \cap \ScptK_{k}} \chi (u) \, du )$ is zero, the integral $\int_{\ScptU \cap \ScptK_{k}} \lambda (u) v \ du = ( \int_{\ScptU \cap \ScptK_{k}} \chi (u) \, du ) \, v$ is zero.  So, $\sigma$  is cuspidal.  Also, by \cite{MoyPr1}, the depth of $\sigma$ is $k$, and therefore (by \cite{MoyPr1}, Theorem 5.2) $\sigma$ cannot contain a non-zero ${\mathcal K}_{k}$-fixed vector.     
\end{proof}

\medskip

A similar result can be shown with the group ${\mathcal K}_{p,k}$ replaced by the group ${\mathcal I}_{p,k^{+}}$.  We explain it in the next subsection where we treat the more general getting of a split simple group.

\bigskip

\subsection{The case of a split simple group}\label{case-split-simple} \ 

\smallskip

Suppose $G$ is a split simple algebraic group defined over $\bbZ_{p}$.  Let $B$ be a Borel subgroup of $G$, $A$ a maximal torus of $B$. Set 
\begin{equation}
\aligned
{\mathscr G} :&= G({\mathbb Q}_p) {\text{\rm{ \ and \ }}}
{\mathcal K} := G({\mathbb Z}_p) \ {\text{\rm{a maximal compact subgroup of ${\mathscr G}$.}}}
\endaligned
\end{equation}  
\noindent We can choose the torus $A \subset B$ so that the point in the building fixed by ${\mathcal K}$ lies in the apartment of $A$.  Then, $B$ determines an Iwahori subgroup ${\mathcal I} \subset {\mathcal K}$.

\medskip

\smallskip

Let $\ScptB ({\mathscr G})$ be the Bruhat-Tits building of ${\mathscr G}$.  Let $C = \ScptB ({\mathscr G})^{\mathcal I}$ be the fixed points of 
of the Iwahori subgroup ${\mathcal I}$.  It is an alcove in $\ScptB ({\mathscr G})$.  Take $x_{0}$ to be the barycenter of $C$, and let $\ell$ be the rank of $G$.  Set 

\begin{equation}\label{moy-prasad-0}
k' \ := \ k+ ({\frac{1}{\ell+1}}) \quad {\text{\rm{and}}} \quad k'' \ := \ k+ ({\frac{2}{\ell+1}})
\end{equation}

\noindent Then, in terms of the Moy-Prasad filtration subgroups, define

\begin{equation}\label{moy-prasad-1}
{\mathcal I}_{k}  \ := \ {\mathscr G}_{{x_{0}},k}  \quad {\text{\rm{and}}} \quad 
{\mathcal I}_{k^{+}}  \ := \ {\mathscr G}_{{x_{0}},k'} \ . 
\end{equation}

\smallskip

\noindent For ${\mathscr G} = SL_{M}(\bbQ_{p})$, the group ${\mathcal I}_{k^{+}}$ is identical with the group defined in \eqref{iwahori-k} by the smae symbol.

\smallskip

Let $\Delta$ and $\Delta^{\text{\rm{aff}}}$ be the simple roots and simple affine roots respectively with respect to the Borel and Iwahori subgroups $B$ and ${\mathcal I}$ respectively.  We recall that every $\alpha \in \Delta$ is the gradient part of a unique root  $\psi \in \Delta^{\text{\rm{aff}}}$.  In this way, we view $\Delta$ as a subset of $\Delta^{\text{\rm{aff}}}$.  We recall 

\begin{equation}
{\text{\rm{ the quotient }}} \ \ {\mathscr G}_{{x_{0}},k'} / {\mathscr G}_{{x_{0}},k''} \ \ {\text{\rm{ is canonically }}} {\underset {\psi \in \Delta^{\text{\rm{aff}}}}{\prod}} U_{(\psi+k)} / U_{(\psi+k+1 )} \ .
\end{equation}

\noindent We further recall that a character $\chi$ of ${\mathscr G}_{{x_{0}},k'} / {\mathscr G}_{{x_{0}},k''}$ is non-degenerate if the restriction of $\chi$ to any $U_{(\psi+k)}$ is non-trivial.  In particular, it is clear there exists a non-degenerate character $\chi$ of ${\mathscr G}_{{x_{0}},k'} / {\mathscr G}_{{x_{0}},k''}$ for any integer $k \ge 0$.  For convenience, we identify a function on ${\mathscr G}_{{x_{0}},k'} / {\mathscr G}_{{x_{0}},k''}$ with its inflation to the group ${\mathscr G}_{{x_{0}},k'}$.

\begin{Lem} \label{iwahori-lemma} Let $p$ be a prime such that $G$ is unramified over $\bbQ_{p}$.    Let ${\mathcal I}_{k^{+}}$ ($k \ge 0$) denote the subgroup in \eqref{moy-prasad-1}, and let $\chi$ be a non-degenerate character of  ${\mathscr G}_{{x_{0}},k'} / {\mathscr G}_{{x_{0}},k''}$.  Then, 
\smallskip
\begin{itemize}
\item[(i)] The inflation of $\chi$ to ${\mathscr G}_{{x_{0}},k'}$, when extended to ${\mathscr G}$ by zero outside ${\mathscr G}_{{x_{0}},k'}$, is a cusp form of ${\mathscr G}$.  
\smallskip
\item[(ii)] For each $k \ge 0$, there exists an irreducible supercuspidal representation $(\rho_{p}, W_{p})$ which has a non-zero ${\mathcal I}_{(k+1)^{+}}$--invariant vector but no non-zero ${\mathcal I}_{(k)^{+}}$--invariant vector.
\end{itemize}
\end{Lem}
  
\begin{proof}  To prove part (i), suppose $x \in {\mathscr G}=G({\mathbb Q}_p)$, and ${\mathscr U}=U({\mathbb Q}_p)$ is the unipotent radical of a proper parabolic subgroup ${\mathscr P}=P({\mathbb Q}_p)$ of ${\mathscr G}$.  We need to show 
\begin{equation}\label{local-cusp-from-1}
\int_{{\mathscr U}} \chi (xu) \ du \ = \ 0 \, . 
\end{equation}

\noindent Take $Q \subset {\mathscr G}$ to be a ${\mathbb Z}_p$-defined parabolic subgroup so that ${\mathscr Q}= Q({\mathbb Q}_p)$ is ${\mathscr G}$ conjugate to ${\mathscr P}$, i.e., $P = gQg^{-1}$, with $g \in {\mathscr G}$.   Let $V$ and ${\mathscr V}$ denote the unipotent radical of $Q$, and its group of $\kv$-rational points.  We use the Iwasawa decomposition ${\mathscr G} = {\mathscr K} {\mathscr Q}$ to write $g$ as $g = k h$.  Then,

\begin{equation}
\aligned
\int_{{\mathscr U}} \psi_{Y}(xu) \ du \ &= \  \int_{{\mathscr V}} \psi_{Y}( \, x \, g v g^{-1} \, ) \ dn \ \ {\text{\rm{($u=gvg^{-1}$)}}}\\
&= \  \int_{{\mathscr V}} \psi_{Y}( \, x \, k h v h^{-1} k \, ) \ du  \\
&= \  c \, \int_{{\mathscr V}} \psi_{Y}( \, x \, k v k^{-1} \, ) \ dv  \ \ {\text{\rm{(for a suitable constant $c$)}}} \\ 
&= \ c \,  \int_{k{\mathscr V}k^{-1}} \psi_{Y}( \, x \, n \, ) \ dn \   .
\endaligned
\end{equation}

\noindent In particular, we can reduce to the case where the parabolic $P$ is a ${\mathbb Z}_p$-defined subgroup of ${\mathscr G}$.   But, then $P$ is ${\mathcal K}$-conjugate to a standard parabolic subgroup of ${\mathscr G}$ with respect to the maximal split torus $A$.  So, we can and do assume $P$ is a standard parabolic.  

\smallskip
Observe that since ${\text{\rm{supp}}}(\chi) = {\mathscr G}_{{x_0},k'}$, to show  \eqref{local-cusp-from-1}, it suffices to take $x \in {\mathscr G}_{{x_0},k'}$.  Then, $xu \in {\mathscr G}_{{x_0},k'}$ if and only if $u \in {\mathscr G}_{{x_0},k'} \, \cap \, {\mathscr U}$, so 

\begin{equation}\label{integral-iwahori-1}
\aligned
\int_{{\mathscr U}} \chi (xu) \ du \ &= \ \int_{{\mathscr G}_{{x_0},k'} \ \cap \ {\mathscr U}} \chi (xu) \ du \\
\endaligned
\end{equation}
\noindent  The intersection ${{\mathscr G}_{{x_0},k'} \, \cap \, {\mathscr U}}$ is a product of affine root subgroups.  Combining this with the fact that $\chi$ is a character, we see that the integral over ${{\mathscr G}_{{x_0},k'} \ \cap \ {\mathscr U}}$ is a product of integrals over the affine root subgroups.  Since $U$ is the radical of a proper standard parabolic subgroup, at least one of the $A$-roots in $U$ is the gradient of an affine root $\psi \in \Delta^{\text{\rm{aff}}}$.  But then 

\begin{equation}\label{integral-iwahori-2}
\aligned
\int_{{\mathscr G}_{{x_0},k'} \ \cap \ {\mathscr U}_{\psi}} \chi (xu) \ du \ = \ 0 \\
\endaligned
\end{equation}

\noindent since $\chi$ is a non-trivial character of ${\mathscr U}_{(\psi +k)} = {\mathscr G}_{{x_0},k'} \ \cap \ {\mathscr U}_{\psi}$.  Thus, $\chi$ is a cusp form.   This completes the proof of part (i).

\medskip

To prove part (ii), let $V_{\chi}$ denote the vector space spanned by left translations of $\chi$.  That $\chi$ is a cusp form of ${\mathscr G}$ means $V_{\chi}$, as a representation of ${\mathscr G}$, is a direct sum of finitely many irreducible cuspidal representations, and by Frobenius reciprocity each irreducible cuspidal representation $\sigma$ which appears when restricted to ${\mathscr G}_{x_{0},k'}$ contains the character $\chi$.  In particular, $\sigma$ contains a non-zero ${\mathscr G}_{x_{0},k''}$-fixed vector, whence a non-zero ${\mathcal I}_{(k+1)^{+}}$-fixed vector.  The fact that $\sigma$ contains the non-degenerate character $\chi$ and $\sigma$ is assumed to be irreducible means it cannot have a ${\mathscr G}_{x_{0},k'}={\mathcal I}_{k^{+}}$-fixed vector.  So (ii) holds.

\end{proof}

\bigskip

\section{Examples of open compact subgroups ${\calF}$ satisfying assumptions
\eqref{assumptions}}\label{sec-3}

\medskip

We produce examples of finite sets ${\calF}$ of open compact subgroups of $G({\bbA}_f)$ satisfying the assumptions \eqref{assumptions}.  

\medskip

Suppose $G = \SL{M}$.
\smallskip 
\begin{itemize}
\item[$\bullet$] Fix a positive integer $D$.  For each positive divisor $d$ of $D$, set $K_d$ as in \eqref{open-compact-SL}.  Then, as a consequence of Proposition \eqref{SL-cusp-forms}, the finite family 
$$
{\calF} \ = \ \{ \, K_d \ \big| \ \ d \, | \, D \ \}
$$
satisfies the assumptions in \eqref{assumptions}.  Whence, Theorem \eqref{intr-thm} applies to this family.   As already mentioned in Proposition \eqref{gamma-one-2} $K_d \cap \SL{M}({\bbQ})$ is the principal congruence subgroup $\Gamma (d)$ of \eqref{principal-congruence}.
\smallskip

\item[$\bullet$] Fix a positive integer $D$.  For each positive divisor $d$ of $D$, set $J_d$ as in \eqref{open-compact-SL}.  Then, as a consequence of Lemma \eqref{iwahori-lemma}, the finite family 
$$
{\calF} \ = \ \{ \, J_d \ \big| \ \ d \, | \, D \ \}
$$
satisfies the assumptions in \eqref{assumptions}.  Whence, Theorem \eqref{intr-thm} applies to this family.   Here, $J_d \cap \SL{M}({\bbQ})$ is the subgroup $\Gamma_1 (d)$ of \eqref{iwahori-congruence}.

\end{itemize}

Recall that we have been assuming $G$ is simply connected, absolutely almost simple over $\bbQ$ and $G_\infty$ is not compact.   Let $S_{f} = \{ p_1, \dots , p_r, p_{r+1}, \dots , p_{r+s} \}$ be primes satisfying the following:

\smallskip

\begin{itemize}
\item[(i)] For $v \notin S_{f}$, the group $G$ is unramified at $v$.   
\smallskip
\item[(ii)] For $v \in S_{f}$, we consider two cases: 
\begin{itemize}
\item[(ii.1)] For $1 \le i \le r$, we are given open compact subgroups ${\mathcal L}_{p_i} \subset G({\bbQ}_{p_i})$.
\smallskip
\item[(ii.2)] For $(r+1) \le i \le (r+s)$, the group $G$ is unramified at $p_i$.
For each $p_i$, take $C_i$ to be an alcove in the Bruhat-Tits building and 
$x({C_i})$ the barycenter of $C_i$.
\end{itemize}
\end{itemize} 

\smallskip

\noindent Fix some exponents $e_{r+1}, \dots , e_{r+s}$, and set
$$
D \ = \ p_{r+1}^{e_{r+1}} \cdots p_{j}^{e_{j}} \cdots p_{r+s}^{e_{r+s}} \ .
$$
For $d = p_{r+1}^{\alpha_{r+1}} \cdots p_{r+s}^{\alpha_{r+s}}$ a divisor of $D$, set 
\begin{equation}\label{open-compact-3}
L_{d} \ := \ {\underset {i=1} {\overset {r} \prod}} \ {\mathcal L}_{p_i} \ \ {\underset {i=(r+1)} {\overset {(r+s)} \prod} } \ G({\bbQ_{p_i}})_{x({C_i}),\alpha_{i}'}  \ \  {\underset {v \notin S_f} {\prod}} \ G({\bbZ}_p) \ .
\end{equation}
Here, $\alpha_{i}'$ is defined as in \eqref{moy-prasad-0}.  Then, $\calF = \{ \ L_d \ \big| \ \  d \, | \, D \ \}$ is a family of open compact subgroups of $G({\bbA}_f)$ satisfying the assumptions \eqref{assumptions}. 
Whence, Theorem \eqref{intr-thm} applies to this family.   

\smallskip

We note that if we had selected a different choice of alcoves $C_{i}^{\bullet}$, then the groups $G({\bbQ_{p_i}})_{x({C_i}),\alpha_{i}'}$ and $G({\bbQ_{p_i}})_{x({C_{i}^{\bullet}}),\alpha_{i}'}$ are conjugate in $G({\bbQ_{p_i}})$, say $g_{p,i} \, G({\bbQ_{p_i}})_{x({C_i}),\alpha_{i}'} g_{p,i}^{-1} = G({\bbQ_{p_i}})_{x({C_{i}^{\bullet}}),\alpha_{i}'}$.  Denote by $L_{d}^{\bullet}$, the open compact subgroup of $G({\bbA}_f)$ obtained in \eqref{open-compact-3} by replacing $G({\bbQ_{p_i}})_{x({C_{i}}),\alpha_{i}'}$ with $G({\bbQ_{p_i}})_{x({C_{i}^{\bullet}}),\alpha_{i}'}$.  In $G({\bbA}_f)$ the element
$$
g = \ {\underset {i=1} {\overset {r} \prod}} \ 1_{G({\bbQ_{p_i}})} \ \ {\underset {i=(r+1)} {\overset {(r+s)} \prod} } \ g_{p,i}  \ \  {\underset {v \notin S_f} {\prod}} \ 1_{G({\bbQ}_p)}
$$ 
conjugates $L_{d}$ to $L_{d}^{\bullet}$.  Since $G({\bbA}_f)$ satisfies strong approximation, $g = g_{\bbQ} h_{L_{d}}$, with $g_{\bbQ} \in G({\bbQ})$ and $h_{L_{d}} \in L_{d}$.  It follows  $L_{d}$ and $L_{d}^{\bullet}$ are conjugate by the element $g_{\bbQ}$.  In particular, the intersections 
$$
L_{d} \cap G({\bbQ}) \ {\text{\rm{\ and \ }}} \ L_{d}^{\bullet} \cap G({\bbQ})
$$
are conjugate by the element $g_{\bbQ}$  in  $G({\bbQ})$.

\bigskip

\section{Some Additional Results for $SL_{M}$}\label{sec-4}

In this section we let $G=SL_M$.  Set $G_{\infty} := SL_{M}(\bbR )$, and $K_{\infty} := \mathrm{SO} (M)$.  

\medskip

We prove some simple properties of the intersection of the principal congruence subgroups $\Gamma (m)$ with $K_\infty$. 

\begin{Lem}\label{nlem-2} Let $m\ge 1$. If $g\in \Gamma(m)$ is a diagonal element, then 
$g_{ii}\in \{\pm 1\}$ for all $i=1, \ldots, M$.
\end{Lem}
\begin{proof} Since  $g_{11}g_{22}\cdots g_{MM}=1$, the claim follows. 
\end{proof}

\vskip .2in 

\begin{Lem}\label{nlem-3} Let $m\ge 3$. If $g\in
 \Gamma(m)$ is a diagonal element, then 
$g=I_{M \times M}$.
\end{Lem}
\begin{proof} By Lemma \ref{nlem-2}, $g_{ii}\in \{\pm 1\}$ for all $i=1, \ldots, M$. Since 
$g_{ii}\equiv 1 \ (mod \ m)$, we obtain $m| (\pm 1 -1)$. Finally, $m\ge 3$ implies that 
$g_{ii}= 1$ for all $i=1, \ldots, M$.
\end{proof}

\vskip .2in

\begin{Lem}\label{nlem-4} Let $m\ge 3$. Then $\Gamma(m)\cap K_\infty$ is the identity subgroup. 
\end{Lem}
\begin{proof} Let $g=(g_{ij})\in  \Gamma(m)\cap K_\infty$. Then, by definition of $\Gamma(m)$, 
$m| g_{ij}$ for $i\neq j$. But $|g_{ij}|\le 1< m$. Hence $g_{ij}=0$ for $i\neq j$. 
Thus, $g$ is a
diagonal element of  $\Gamma(m)$. Hence, Lemma \ref{nlem-3} implies the claim.
\end{proof}

\bigskip
We now use Proposition \ref{SL-cusp-forms}.

\begin{Lem}\label{nlem-1}
Suppose $i>1$ is an integer.  Consider the open compact subgroups ${\mathcal K}_{p,i-1}$ and ${\mathcal K}_{p,i}$ in $SL_{M}(\bbQ )$ (notation as in \eqref{iwahori-k}).  Then, there exists a function $f_p$ on $SL_{M}(\bbQ )$ so that:
\begin{itemize}
\item[(i)] $f_P$ is a cusp form with support ${\mathcal K}_{p,i-1}$,
\item[(ii)] $f_p$ is right ${\mathcal K}_{p,i}$-invariant, 
\item[(iii)] $f_p(1) \neq 0$.
\end{itemize}
\end{Lem}

\begin{proof}  Apply Proposition \ref{SL-cusp-forms}.
\end{proof}

\bigskip 
 
We now use Lemma \ref{nlem-1} to considerably improve Lemma \ref{lem-2}.

\smallskip 

\begin{Lem}\label{nlem-5} Suppose $n>1$ is an integer.  Let $T$ denote the set of primes dividing $n$, and suppose 
\begin{equation}\label{condition-3p} 
n \, \ge \, 3 \, {\underset {p\in T} \prod} \, p \ .
\end{equation}
Then, for any $\delta\in \hat{K}_\infty$, 
there exists $f_\infty\in  C_c^\infty(G_\infty)$ so that 
the following holds:
\begin{itemize}
\item[(i)] $E_\delta(f_\infty)=f_\infty$.
\item[(ii)] For 
$$
f \ =\ f_\infty \, \otimes_p \, f_p
$$

\noindent where  $f_p$  is as in Lemma \ref{nlem-1}  when $p \in T$, and $f_{p}=char_{SL_{M}({\bbZ_{p})}}$ when $p\notin T$, then the 
Poincar\' e series $P(f)$ and its restriction to
 $G_\infty$, which is a Poincar\' e series for $\Gamma(n)$, are   non--zero.
\item[(iii)]  $E_\delta(P(f))=P(f)$, $E_\delta(P(f)|_{G_\infty})=P(f)|_{G_\infty}$,  and 
$P(f)$ is right--invariant under $K_{n}$.
\item[(iv)] The support of $P(f)|_{G_\infty}$ is contained in a set of the form 
$\Gamma(n)\cdot C$, where 
$C$ is a compact set which is right--invariant under 
$K_\infty$, and  $\Gamma(n)\cdot C$ is not whole $G_\infty$. 
\item[(v)] $P(f)$ is cuspidal and  $P(f)|_{G_\infty}$ is $\Gamma(n)$--cuspidal.
\end{itemize}
\end{Lem}
\begin{proof} We use Lemmas \ref{lem-1} and \ref{lem-2}. We also use the notation introduced 
in the paragraph before and in Lemma \ref{nlem-1}. This meets all assumptions of Lemma 
\ref{lem-2} (with $L=K_{n}$). We let $S=T\cup \{\infty\}$. 

We need to study the intersection (\ref{e-3}). In our 
case it is given by 

\begin{equation} \label{n-2}
\Gamma_S \cap \left[K_\infty \times \prod_{p\in T} \supp{\ (f_p)}\right].
\end{equation}
Thanks to Lemma \ref{nlem-1}, this is a subset of 
$$
\Gamma_S \cap \left[K_\infty \times \prod_{p\in T} {\mathcal K}_{p,\nu_p(n)-1}\right].
$$
But projecting down to the first factor, this intersection becomes
$$
K_\infty\cap \Gamma(n/\prod_{p\in T}p).
$$ 

\noindent By Lemma \ref{nlem-4}, and our assumption $n \, \ge \, 3 {\underset {p\in T} \prod} \, p$,
it is trivial.  Whence, (\ref{n-2}) consists of the identity only. 
In particular, in (\ref{e-3}), we have $K_\infty\cap \Gamma=\{1\}$, $l=1$,  
$\gamma_1=1$, and $c_1\neq 0$.  We remark that $\Gamma=\Gamma(n)$ (see (\ref{e-20})). 

\smallskip

Next, by Lemma \ref{lem-2}, we need to study the map (\ref{e-4}). Thanks to the above computations, this map is $\alpha\mapsto c_1\cdot \alpha$. Hence, it is essentially the identity. 
Now, (i)--(iv) of the lemma follow from (i)--(iv) from Lemma \ref{lem-2}
for any $K_\infty$--type $\delta$. 
Finally, (v) follows from Lemma \ref{lem-1}, and (\cite{Muic1}, Proposition 5.3). 
\end{proof}

\bigskip

We now prove the main result of this section. It is analogous and  generalizes the main result of \cite{Muic3} (see \cite{Muic3}, Theorem 0-1).

\bigskip

\begin{Thm}\label{nthm}  Suppose $n >1$ is an integer and satisfies the condition \eqref{condition-3p} that $n \ge 3 {\underset {p \in T} \prod p }$.  Then, for any $\delta\in \hat{K}_\infty$, the orthogonal complement of 
$$
\sum_{\substack{m| n\\ m< n }} L^2_{cusp}(\Gamma(m)\backslash G_\infty)
$$ 
in 
$L^2_{cusp}(\Gamma(n)\backslash G_\infty)$ contains a direct sum of infinitely many 
inequivalent irreducible unitary  representations of $G_\infty$ all containing $\delta$.
\end{Thm}
\begin{proof} The (finite) family of open compact subgroups 
$$
{\mathcal F} = \{ \ K_{m} \ |  \ \ 1 \le m\le n, \  m|n \ \}
$$
meet all the assumptions of \ref{assumptions} with the group
$K_{n}$ contained in all the other $K_{m}\in {\mathcal F}$.
(See Section \ref{sec-3}.) Now, the proof is the same as the proof of Theorem 
\ref{intr-thm}. We leave the details to the
reader.

\end{proof}

\vskip .2in

\begin{Cor}\label{ncor-1} Suppose $n >1$ satisfies $n \ge 3 {\underset {p \in T} \prod p }$.  Then the orthogonal complement of 
$$
\sum_{\substack{m| n\\ m< n }} L^2_{cusp}(\Gamma(m)\backslash G_\infty)
$$ 
in 
$L^2_{cusp}(\Gamma(n)\backslash G_\infty)$ contains a direct sum of infinitely 
many inequivalent irreducible unitary 
$K_\infty$--spherical representations of $G_\infty$.
\end{Cor}

\vskip .2in 

\begin{Cor}\label{ncor-2}  Suppose $n >1$ satisfies $n \ge 3 {\underset {p \in T} \prod p }$.  Then, for every $\delta\in 
\hat{K}_\infty$, the orthogonal complement of 
$$
\sum_{\substack{m| n\\ m< n }} L^2_{cusp}(\Gamma(m)\backslash G_\infty)
$$ 
in 
$L^2_{cusp}(\Gamma(n)\backslash G_\infty)$ contains a direct sum of infinitely many inequivalent irreducible unitary 
representations of $G_\infty$ all contaning $\delta$ which are not in the discrete series or in the limits of 
discrete series for $G_\infty$.
\end{Cor}
\begin{proof} As in (\cite{Muic3}, Proposition 4.2).
\end{proof}

\vskip .2in
Let $P_\infty=M_\infty A_\infty N_\infty$ be the Langlands decomposition of a minimal parabolic subgroup  of $G_\infty$. We let $\fraka_\infty$ be the real 
Lie algebra of $A_\infty$ and   $\fraka^*_\infty$ its complex dual. 
We use Vogan's theory of minimal $K_\infty$--types (\cite{vog}, \cite{vog-1}). Any 
$\epsilon\in \hat{M}_\infty$ is fine (\cite{vog-1}, Definition 4.3.8).

\smallskip

Let $\epsilon\in \hat{M}_\infty$. Following (\cite{vog-1}, Definition 4.3.15), we let $A(\epsilon)$ be the
set of  $K_\infty$--types $\delta$ such that $\delta$ is fine (\cite{vog-1}, Definition 4.3.9) and  $\epsilon$ occurs in 
$\delta|_{M_\infty}$. Applying (\cite{vog-1}, Theorem  4.3.16), we obtain that $A(\epsilon)$ is not empty and  
for $\delta \in A(\epsilon)$,  we have the following:
\begin{equation}\label{K-ind}
\delta|_M= \oplus_{\epsilon'\in  \{w(\epsilon); \ w\in W\}} \epsilon',
\end{equation}
where $W=N_{K_\infty}(A_\infty)/M_\infty$ is the Weyl group of $A_\infty$ in $G_\infty$. Since the restriction map implies 
$\Ind_{M_\infty A_\infty N_\infty}^{G_\infty}(\epsilon \otimes \exp{\nu(\ )})\simeq \Ind_{M_\infty}^{K_\infty}(\epsilon)$
as $K_\infty$--representations, by Frobenius reciprocity and (\ref{K-ind}) we see for every $\nu \in \fraka^*_\infty$  there exists a unique irreducible
subquotient $J_{\epsilon\otimes \nu}(\delta)$ of
$\Ind_{M_\infty A_\infty N_\infty}^{G_\infty}(\epsilon \otimes \exp{\nu(\ )})$
containing the $K_\infty$--type $\delta$.  

\vskip .2in 

One important example is the case $\epsilon=\triv_{M_\infty}$. Then $\mu=\triv_{K_\infty}\in A(\triv_M)$, and 
 $J_{\epsilon\otimes \nu}(\delta)$ is the unique $K_\infty$--spherical irreducible subquotient of 
$\Ind_{M_\infty A_\infty N_\infty }^{G_\infty}(\epsilon \otimes \exp{\nu(\ )})$.

\vskip .2in 
\begin{Cor} \label{ncor-3}  Suppose $n >1$ satisfies $n \ge 3 {\underset {p \in T} \prod p }$.  Let $\epsilon\in \hat{M_\infty}$.  Then, for every
$\delta \in A(\epsilon)$, there exist infinitely many $\nu \in \fraka^*$ such that 
$J_{\epsilon\otimes \nu}(\delta)$  appears in the orthogonal complement of 
$$
\sum_{\substack{m| n\\ m< n }} L^2_{cusp}(\Gamma(m)\backslash G_\infty)
$$ 
in 
$L^2_{cusp}(\Gamma(n)\backslash G_\infty)$.
\end{Cor}
\begin{proof}
 As in (\cite{Muic3}, Theorem 4.8).
\end{proof}

\end{document}